\documentclass[12pt,oneside]{amsart}
\usepackage{amssymb,verbatim,enumerate,ifthen}
\usepackage{cite}
\usepackage[mathscr]{eucal}
\usepackage[utf8]{inputenc}
\usepackage[T1]{fontenc}
\usepackage{esint}
\usepackage{marginnote}
\usepackage{color}
\usepackage[marginparwidth=2.4cm]{geometry}
\usepackage{comment}

\newtheorem*{thm}{Theorem}
\newtheorem*{cor}{Corollary}

\newtheorem{lem}{Lemma}

\theoremstyle{definition}

\def\eq#1{{\rm(\ref{#1})}}
\def\Eq#1#2{\ifthenelse{\equal{#1}{*}}
  {\begin{equation*}\begin{aligned}[]#2\end{aligned}\end{equation*}}
  {\begin{equation}\begin{aligned}[]\label{#1}#2\end{aligned}\end{equation}}}

\def\E{\mathscr{E}}

\def\P{\mathscr{P}}
\def\O{\mathcal{O}}
\newcommand\R{\mathbb{R}}

\newcommand\N{\mathbb{N}}

\def\oR{\overline{\R}}

\title
{New bounds for the ratio of power means}

\author{Zsolt P\'ales}
\address{Institute of Mathematics, University of Debrecen, Pf.\ 400, 4002 Debrecen, Hungary}
\email{pales@science.unideb.hu}

\author{Pawe\l{} Pasteczka}
\address{Institute of Mathematics, Pedagogical University of Krak\'ow,  Podchor\k{a}\.{z}ych str 2, 30-084 Krak\'ow, Poland}
\email{pawel.pasteczka@up.krakow.pl}

\thanks{The research of the first author was supported by the EFOP-3.6.1-16-2016-00022 and the EFOP-3.6.2-16-2017-00015 projects. These projects are co-financed by the European Union and the European Social Fund.}

\keywords{H\"older means; power means; ratio bounds}

\subjclass[2010]{26E60, 26D15, 26D07}

\begin{document}
\begin{abstract}
We show that for real numbers $p,\,q$ with $q<p$, and the related power means $\mathscr{P}_p$, $\mathscr{P}_q$, the inequality
$$\frac{\mathscr{P}_p(x)}{\mathscr{P}_q(x)} \le \exp \bigg( \frac{p-q}8 \cdot \bigg(\ln\bigg(\frac{\max x}{\min x}\bigg)\bigg)^2 \:\bigg)$$
holds for every vector $x$ of positive reals. Moreover we prove that, for all such pairs $(p,q)$, the constant $\tfrac{p-q}8$ is sharp.
\end{abstract}
\maketitle

\section{Introduction}

A classical result states that if $(\P_p)_{p \in \oR}$, $\P_p \colon \bigcup_{n=1}^\infty \R_+^n \to \R$ is the extended family of power means defined by
\Eq{*}{
\P_p (v_1,\dots,v_n):= 
\begin{cases} 
\min(v_1,\dots,v_n) &\quad \text{ if }p = -\infty,\\[1mm]
\Big(\dfrac{v_1^p+\cdots+v_n^p}{n}\Big)^{1/p} &\quad \text{ if } p\in\R\setminus\{0\}, \\[2mm]                                                          
\sqrt[n]{v_1\cdots v_n} &\quad \text{ if } p = 0,\\
\max(v_1,\dots,v_n) &\quad \text{ if }p = +\infty,
\end{cases}
}
then for every nonconstant vector $v$ of positive entries, the mapping $\oR \ni p \mapsto \P_p(v) \in [\min v ,\max v]$ is a continuous bijection.
Furthermore, the three-variable ratio function $R \colon \oR \times \oR \times \bigcup_{n=1}^\infty \R_+^n \to \R$ given by $R(p,q,v):=\P_p(v)/\P_{q}(v)$ is continuous on $\oR \times \oR \times \R_+^n$ for all fixed $n\in\N$.

The well-known properties of power means imply a number of assertions concerning this function. For example, $R$ is strictly increasing in the first variable and strictly decreasing in the second variable unless $v$ is a constant vector. Yet more difficult one is due to Cargo and Shisha \cite{CarShi62}, who established the following upper estimate which is valid for all nonconstant vectors $v \in \bigcup_{n=1}^\infty \R_+^n$ and all $p,q \in \R \setminus \{0\}$ with $q<p$ 
 \Eq{CS}{
R(p,q,v) \le 
{\bigg(\dfrac{q(\gamma^p-\gamma^q)}{(p-q)(\gamma^q-1)}\bigg)^{1/p}}{\bigg(\dfrac{p(\gamma^p-\gamma^q)}{(p-q)(\gamma^p-1)}\bigg)^{-1/q}},\quad\text{where }\gamma:=\frac{\max v}{\min v}.
 }
In the most classical particular case, when $(p,q)=(1,-1)$, that is, when the means are the arithmetic and harmonic means, the above upper estimate reduces to Kantorovich's celebrated inequality \cite[pp. 142--143]{Kan48}:
\Eq{*}{
  R(1,-1,v)\leq \frac{(\gamma+1)^2}{4\gamma}, \qquad\text{where }\gamma:=\frac{\max v}{\min v}.
}
Let us underline that, as both sides are continuous functions of $p$ and $q$, we can extend this inequality to the case $pq=0$. In fact, this was done in the paper \cite{CarShi62}, but we will not use this result explicitly.

On other hand, the Taylor expansion and a few general results concerning means (for example \cite[Theorem 8.8]{BorBor87}) imply that, for (fixed) $p,q \in \R$, we have 
\Eq{*}{
\sup \{R(p,q,x) \colon \max x\leq\gamma\min x\} = 1+\O((\gamma-1)^2)
\qquad\text{for }\gamma\approx1.
}
This observation motivates us to look for upper estimations of this type.

As a corollary, we majorize the difference between two so-called exponential means. For a given parameter $p\in \oR$ a \emph{$p$-th exponential mean} $\E_p \colon \bigcup_{n=1}^\infty \R^n \to \R$ is given by
\Eq{*}{
\E_p(v_1,\dots,v_n):=
\begin{cases}
\dfrac{1}{p} \ln \bigg( \dfrac{e^{pv_1}+\dots+e^{pv_n}}{n} \bigg) & \text{ for }p \in \R \setminus\{0\}, \\[5mm]
\dfrac{v_1+\dots+v_n}n  & \text{ for }p = 0,
\end{cases}
}
with a natural limit-type extension $\E_{-\infty}:=\min$ and $\E_{+\infty}:=\max$. We underline that this family is closely related to power means, as we have the following conjugacy property
\Eq{conj}{
\E_p(v_1,\dots,v_n)=\ln \big( \P_p(\exp v_1,\dots,\exp v_n) \big) \qquad \text{ for }p \in \oR, n \in \N \text{ and } v \in\R^n.
}

\section{Results}

In this section our main result is presented which provides an upper estimate for the ratio of two power means. It is preceded by three technical although elementary lemmas. At the very end, we also derive an estimation for the difference of two exponential means.

\begin{lem}\label{sub}
Let $I \subset [0,+\infty)$, and $f \colon I \to \R$ be a nonincreasing function. Then the mapping $I \ni x \mapsto xf(x)$ is subadditive. 
\end{lem}
\begin{proof}
Fix $x,y\in I$ with $x+y\in I$. Then, as $f$ is nonincreasing, we have $f(x+y)\le f(x)$. This implies $xf(x+y)\le xf(x)$. Similarly $yf(x+y) \le yf(y)$. Summing up both of these inequalities, we obtain $(x+y)f(x+y) \le xf(x)+yf(y)$, which proves the subadditivity of the mapping $x\mapsto xf(x)$.
\end{proof}

\begin{lem}\label{subf}
The function $f \colon \R \to \R$ given by
 \Eq{E:f}{
f(x):=\begin{cases}
             \ln\bigg( \dfrac{\sinh x}x \bigg) -\dfrac{x^2}6 & x \ne 0, \\
             0 & x=0
            \end{cases}
}
is continuous on $\R$, is even, and is strictly decreasing on $[0,+\infty)$. Consequently, the mapping $[0,+\infty) \ni x\mapsto xf(x)$ is subadditive.
\end{lem}
\begin{proof}
 It is easy to verify that $f$ is continuous and even. To show its strict decreasingness on $[0,+\infty)$, observe that
 \Eq{*}{
f'(x)=\frac{\cosh x}{\sinh x} - \frac{1}{x} -\frac{x}{3}.
 }
Thus, it suffices to show that $f'(x)<0$ holds for all $x \in \R_+$. Equivalently, we have to verify that $(3+x^2)\sinh x < 3x\cosh x$. This latter inequality can be proved by term-wise comparison of the corresponding Taylor series. The last assertion is implied by Lemma~\ref{sub}.
\end{proof}

\begin{lem} \label{fuu}
Let $f \colon \R \to \R$ be given by \eq{E:f}. Then 
\Eq{E:uv}{
\frac {f(q)}p-\frac {f(p)}q+ \bigg( \frac 1q - \frac 1p \bigg) f(p-q) \ge 0
}
for all $p,\,q \in \R \setminus\{0\}$ with $q \le p$. For $p \le q$, the inequality \eq{E:uv} is reversed.
\end{lem}
\begin{proof}
For $p=q$ this statement is trivial. For $q < p$, consider the three cases:
 \Eq{*}{
 (\alpha):\quad 0<q<p, \qquad (\beta):\quad q< 0 <p, \qquad (\gamma):\quad q<p<0\:.
 }
In each of these cases, as $f$ is even, \eq{E:uv} is equivalent to 
\begin{align}
pf(p)&\le qf(q) + (p-q) f(p-q), \tag{$\alpha$}\\
(p-q)f(p-q)&\le pf(p) + (-q) f(-q), \tag{$\beta$}\\
(-q) f(-q) &\le (-p) f(-p) + (p-q) f(p-q). \tag{$\gamma$}
\end{align}
These statements are consequences of Lemma~\ref{subf}, that is, of the subadditivity of the mapping $[0,+\infty) \ni x \mapsto xf(x)$.
The proof for the second case is completely analogous.
\end{proof}

Now we proceed to the main result of the present note.
\begin{thm}
For $p,q \in \R$ with $q\le p$, we have
\Eq{E:thm}{
\frac{\P_p(v)}{\P_q(v)} \le \exp \bigg( \frac{p-q}8 \bigg(\ln\Big(\frac{\max v}{\min v}\Big)\bigg)^2 \:\bigg) \qquad \text{ for all } v \in \bigcup_{n=1}^\infty \R_+^n.
}
Moreover, the constant $\frac{p-q}8$ in the above inequality is sharp. 
\end{thm}

\begin{proof}
First observe that, for a constant vector $v$, the inequality \eq{E:thm} holds with equality. From now on let $v \in \bigcup_{n=1}^\infty \R_+^n$ be a fixed, nonconstant vector. Denote $\gamma:=\frac{\max v}{\min v}$. Then $\gamma>1$ and hence $\ln\gamma>0$.

For the sake of brevity, denote the right hand side of the Cargo--Shisha inequality \eq{CS} by $K$. To prove that inequality \eq{E:thm} is valid for all $p,q \in \R$ with $q \le p$, observe that both sides are continuous functions of $(p,q)$. Therefore, without loss of generality, we may suppose that $p,q \in \R \setminus \{0\}$ and $p \ne q$.

The proof is based on the following identity: 
\Eq{*}{
K&= \Bigg( \dfrac{ q \gamma^{\frac{p+q}{2}} \big(\gamma ^{\frac{p-q}2}-\gamma^{\frac{q-p}2}\big)}{(p-q)\gamma^{\frac{q}{2}}(\gamma^{\frac{q}{2}}-\gamma^{-\frac{q}{2}})} \Bigg)^{\frac{1}{p}}
\Bigg( \dfrac{(p-q)\gamma^{\frac{p}{2}}(\gamma^{\frac{p}{2}}-\gamma^{-\frac{p}{2}})}{ p \gamma^{\frac{p+q}{2}} \big(\gamma ^{\frac{p-q}2}-\gamma^{\frac{q-p}2}\big)} \Bigg)^{\frac{1}{q}}\\
&= 
\bigg(\frac{\sinh(\tfrac{p}{2}\ln \gamma)}{\tfrac{p}{2}\ln \gamma}\bigg)^{\frac{1}{q}}
\bigg(\frac{\sinh(\tfrac{q}{2}\ln \gamma)}{\tfrac{q}{2}\ln \gamma}\bigg)^{-\frac{1}{p}}
\bigg(\frac{\sinh(\tfrac{p-q}{2}\ln \gamma)}{\tfrac{p-q}{2}\ln \gamma}\bigg)^{\frac{1}{p}-\frac{1}{q}}
}
If we now substitute $p_0:=\tfrac{p}2 \ln \gamma$ and $q_0:=\tfrac{p}2 \ln \gamma$, then we have
\Eq{*}{
K^{\tfrac{2}{\ln\gamma}}=\bigg(\frac{\sinh p_0}{p_0}\bigg)^{\frac1{q_0}}
\bigg(\frac{\sinh q_0}{q_0}\bigg)^{-\frac1{p_0}}
\bigg(\frac{\sinh(p_0-q_0)}{p_0-q_0}\bigg)^{\frac1{p_0}-\frac1{q_0}}.
}
Taking logarithm side by side and then using the notation \eq{E:f}, we arrive at
\Eq{*}{
\frac{2\ln K}{\ln \gamma}&=\frac{1}{q_0} \Big( f(p_0)+ \frac{p_0^2}6\Big)- \frac{1}{p_0} \Big( f(q_0)+ \frac{q_0^2}6\Big)+
\Big( \frac{1}{p_0}-\frac{1}{q_0}\Big)\Big( f(p_0-q_0)+ \frac{(p_0-q_0)^2}6\Big)\\
&=\frac{f(p_0)}{q_0}-\frac{f(q_0)}{p_0}+\Big( \frac{1}{p_0}-\frac{1}{q_0}\Big) f(p_0-q_0)+\frac{p_0^3-q_0^3+(q_0-p_0)^3}{6p_0q_0}.
}
But, in view of Lemma~\ref{fuu} (with $p:=p_0$, $q:=q_0$ and observing that $q_0<p_0$), we obtain 
\Eq{*}{
\frac{f(p_0)}{q_0}-\frac{f(q_0)}{p_0}+\Big( \frac{1}{p_0}-\frac{1}{q_0}\Big) f(p_0-q_0) \le 0.
}
Furthermore, we also have
\Eq{*}{
\frac{p_0^3-q_0^3+(q_0-p_0)^3}{6p_0q_0}
=\frac{p_0-q_0}{2}.
}
Thus, combining above statements,
\Eq{*}{
\frac{2\ln K}{\ln \gamma} \le \frac{p_0-q_0}2=\frac{\ln \gamma}{4} (p-q).
}
The latter inequality is equivalent to $K\le \exp \big( \frac{p-q}{8} (\ln \gamma)^2\big)$ which, in view of \eq{CS} and the definition of $K$, completes the proof of \eq{E:thm}.

To verify the sharpness of the constant, let $q<p$ with $pq\neq0$. Assume that the inequality \eq{E:thm} holds with a real constant $C$ (instead of $\frac{p-q}8$) for all $v\in\R_+^2$. Then, with the substitution $v:=(e^t,e^{-t})$, where $t\in\R$, and taking the logarithm side by side, this inequality implies that
\Eq{*}{
  \frac{\ln\cosh(tp)}{p}-\frac{\ln\cosh(tq)}{q}
  \leq 4Ct^2 \qquad(t\in\R).
}
Therefore, the function $t\mapsto \frac{\ln\cosh(tp)}{p}-\frac{\ln\cosh(tq)}{q}- 4Ct^2$ has a maximum at $t=0$, which implies that its second derivative is nonpositive at $t=0$. This implies that $p-q\leq 8C$.
\end{proof}

Using the conjugacy principle \eq{conj}, we immediately obtain a similar statement concerning exponential means.
\begin{cor}
Let $p,q \in \R$ with $q\le p$. Then 
\Eq{*}{
\E_p(v)-\E_q(v) \le \frac{p-q}8 (\max v-\min v)^2 \qquad \text{ for all } v \in \bigcup_{n=1}^\infty \R^n.
}
Moreover, the constant $\frac{p-q}8$ is sharp.  
\end{cor}


\end{document}